\documentclass[12pt]{amsart}

\usepackage{graphics, amsmath, amsfonts, amssymb, amsthm, amscd, mathrsfs, color}

\input xy
\xyoption{all}

\newtheorem{theorem}{Theorem}[section]
\newtheorem{proposition}[theorem]{Proposition}
\newtheorem{lemma}[theorem]{Lemma}
\newtheorem{corollary}[theorem]{Corollary}

\theoremstyle{definition}

\newtheorem{remark}[theorem]{Remark}
\numberwithin{equation}{section}

\newcommand{\A}{{\mathscr A}}
\newcommand{\C}{{\mathbb C}}
\newcommand{\G}{{\mathscr G}}
\renewcommand{\H}{{\mathscr H}}

\renewcommand{\O}{{\mathscr O}}
\newcommand{\R}{{\mathbb R}}

\newcommand{\Aut}{{\operatorname{Aut\,}}}

\newcommand{\inv}{{^{-1}}}

\renewcommand{\phi}{\varphi}
\newcommand{\id}{\mathrm{id}}
\newcommand{\GL}{\mathrm{GL}}

\newcommand{\Iso}{{\operatorname{Iso}}}
\newcommand{\git}{/\!\!/}
\newcommand{\Q}{{\mathscr Q}}

\title{An equivariant parametric Oka principle \\ for bundles of homogeneous spaces}

\author{Frank Kutzschebauch, Finnur L\'arusson, Gerald W.~Schwarz}

\address{Frank Kutzschebauch, Institute of Mathematics, University of Bern, Sidlerstrasse 5, CH-3012 Bern, Switzerland}
\email{frank.kutzschebauch@math.unibe.ch}

\address{Finnur L\'arusson, School of Mathematical Sciences, University of Adelaide, Adelaide SA 5005, Australia}
\email{finnur.larusson@adelaide.edu.au}

\address{Gerald W.~Schwarz, Department of Mathematics, Brandeis University, Waltham MA 02454-9110, USA}
\email{schwarz@brandeis.edu}

\dedicatory{Dedicated to Peter Heinzner on the occasion of his sixtieth birthday}



\date{20 December 2016}

\thanks{F.~Kutzschebauch was partially supported by Schweizerischer Nationalfond grant 200021-140235/1.  F.~L\'arusson was partially supported by Australian Research Council grant DP150103442.  Much of the work on this paper was done at the Centre for Advanced Study at the Norwegian Academy of Science and Letters.  The authors would like to warmly thank the Centre for hospitality and financial support.  F.~Kutzschebauch and G.~W.~Schwarz would also like to thank the University of Adelaide for hospitality and the Australian Research Council for financial support.}

\begin{document}

\begin{abstract}    
We prove a parametric Oka principle for equivariant sections of a holomorphic fibre bundle $E$ with a structure group bundle $\G$ on a reduced Stein space $X$, such that the fibre of $E$ is a homogeneous space of the fibre of $\G$, with the complexification $K^\C$ of a compact real Lie group $K$ acting on $X$, $\G$, and $E$.  Our main result is that the inclusion of the space of $K^\C$-equivariant holomorphic sections of $E$ over $X$ into the space of $K$-equivariant continuous sections is a weak homotopy equivalence.  The result has a wide scope; we describe several diverse special cases.  We use the result to strengthen Heinzner and Kutzschebauch's classification of equivariant principal bundles, and to strengthen an Oka principle for equivariant isomorphisms proved by us in a previous paper.
\end{abstract}

\maketitle
\tableofcontents

\section{Introduction}  \label{sec:introduction}

\noindent
Let $X$ and $E$ be complex spaces and $E\to X$ be a holomorphically locally trivial fibre bundle with fibre $F$ (in this paper, all complex spaces are assumed to be reduced).  Take the structure group of $E$ to be a complex Lie subgroup $G$ of the biholomorphism group of $F$, so $E$ is defined by a holomorphic cocycle with respect to some open cover of $X$ with values in $G$.  We call $E$ a holomorphic $G$-bundle.  The same cocycle defines a holomorphic principal $G$-bundle $P$ such that $E$ may be identified with the twisted product $P\times^G F$, where we take $G$ to act on $P$ from the right and on $F$ from the left.

Let $L$ be a complex Lie group.  \textbf{Our first main goal} is to prove a parametric Oka principle for equivariant sections of $E$, when $L$ acts holomorphically on $X$ and $E$, so that the projection $E\to X$ is equivariant.  For this to be possible, we need to add four assumptions to the very general hypotheses that we have stated so far:
\begin{enumerate}
\item  $X$ is Stein.  This is a necessary assumption in Oka theory.
\item  $G$ acts transitively on $F$ (so $F$ is a complex homogeneous space, in particular smooth).  This assumption strengthens the relationship between $E$ and $P$.  Namely, $E$ may be identified with the quotient bundle $P/H$, where $H$ is a closed complex Lie subgroup of $G$, such that $F$ is biholomorphic to the quotient $G/H$ of left cosets.
\item  $L$ is reductive, that is, $L$ is the complexification $K^\C$ of a compact real Lie group $K$.  This assumption ensures a good structure theory for holomorphic actions of $L$ on $X$, similar to the structure theory of continuous actions of $K$.
\item  We need to suitably restrict the way that $L$ acts on $E$, so as to take the $G$-structure into account.
\end{enumerate}

Assumption (4) requires some discussion.  An action of a group on an object is a homomorphism from the group into the automorphism group of the object.  We want to allow the largest possible group of automorphisms of $E$, so as to have the strongest possible theorem.  It is important to be able to use the theory of principal bundles, so we want the action on $E$ to lift to an action on $P$.  Thus we ask:  What is the biggest group of automorphisms of $P$ that we can reasonably work with?

The simplest and most commonly studied automorphisms of $P\overset\pi\to X$ are the $G$-bundle maps, that is, pairs of maps $\phi:P\to P$ and $f:X\to X$ such that $f\circ\pi=\pi\circ\phi$ and $\phi$ is $G$-equivariant.  So we could have a left action of $L$ on $P$ commuting with the right action of $G$.  Such an action descends to $E$.  A more general action is obtained by twisting by a homomorphism from $L$ into the automorphism group of $G$.  Such an action descends to $E$ if the automorphisms of $G$ in the image of $L$ preserve $H$.  

An even more general action is obtained by considering $P$ to have the trivial group bundle $\G=X\times G\to X$ as its structure group bundle and twisting by an action of $L$ on $\G$.  More precisely, we choose a complex Lie subgroup $A$ of the Lie automorphism group of $G$ as the structure group of $\G$ and allow $L$ to act on $\G$ by automorphisms of $\G$ that act on the fibres of $\G$ by elements of $A$.  Now the allowable automorphisms of $P$ consist of compatible maps $\phi:P\to P$ and $f:X\to X$ such that for all $x\in X$, $p\in P_x$, and $g\in G$,
\[ \phi(p\cdot g)=\phi(p)\cdot\alpha_x(g) \in P_{f(x)}, \]
where $\alpha_x:G_x\to G_{f(x)}$ is a group isomorphism obtained by restricting an automorphism of $\G$.  In other words, the right action map $P\times G=P\times_X\G\to P$ is $L$-equivariant.  The action descends to $E$ if $A$ preserves $H$.  

We want to allow such actions, so we might as well take $P$ to be a generalised principal bundle with an arbitrary structure group bundle $\G$.  Another reason to do so is that \textbf{our second main goal} is to strengthen the classification theorem for principal $\G$-bundles with an $L$-action due to Heinzner and Kutzschebauch \cite{Heinzner-Kutzschebauch}.  Generalised principal bundles arise naturally, even if one is only interested in ordinary principal bundles.  If $P_1$ and $P_2$ are ordinary principal $G$-bundles, then the automorphism bundle $\Aut P_2$ is a group bundle with fibre $G$, typically nontrivial, and the bundle of isomorphisms $\Iso(P_1,P_2)$ is naturally viewed as a principal bundle with structure group bundle $\Aut P_2$.  Therefore, if we are interested in the classification theory of ordinary principal bundles, we need to study sections of generalised principal bundles, so we might as well work in the context of generalised principal bundles from the outset.

Our setting, then, is as follows.  Let $X$ be a reduced Stein space.  Let $G$ be a complex Lie group and $\G$ be a holomorphic group bundle on $X$ with fibre $G$.  By definition, $\G$ is defined by a holomorphic cocycle with respect to some open cover of $X$ with values in a complex Lie subgroup $A$ of the Lie automorphism group of $G$.  We call $A$ the structure group of $\G$ and we call $\G$ a holomorphic group $A$-bundle.  Let $\H$ be a holomorphic group subbundle of $\G$, whose fibre is a closed subgroup $H$ of $G$, so $\G$ may in fact be defined by a holomorphic cocycle with values in the group of Lie automorphisms of $G$ that preserve $H$.  Thus we assume that $A$ preserves $H$.

Let $P$ be a holomorphic principal bundle on $X$ with structure group bundle $\G$ acting from the right---we call $P$ a principal $\G$-bundle---and let $E$ be the quotient bundle $P/\H$.  Then $E$ is a holomorphic fibre bundle on $X$ with fibre $G/H$ (left cosets) and structure group bundle $\G$ acting on the fibre from the left.  Each fibre of $\G$ acts transitively on the fibre of $E$.  We call $E$ a homogeneous $\G$-bundle.  The principal bundle $P$ is defined by a holomorphic $\G$-valued cocycle, which tells us how to form $P$ by glueing together pieces of $\G$ over an open cover of $X$.  The same cocycle encodes how $E$ may be constructed from the quotient bundle $\G/\H$ (left cosets).  Note that the action of $\G$ on $P$ need not descend to an action on $E$ (right multiplication does not respect left cosets).

Recall that complexification defines a bijection from compact real Lie groups to reductive complex Lie groups.  Let $K$ be a compact real Lie group with complexification $K^\C$.  Let $K^\C$ act holomorphically on $X$, and holomorphically and compatibly on $\G$ by group $A$-bundle maps (which preserve $\H$).  This means that $K^\C$ acts on the fibres of $\G$ by elements of $A$, which makes sense because each fibre of $\G$ is canonically identified with $G$ modulo $A$.  Let $K^\C$ also act holomorphically and compatibly on $P$ such that the action map $P\times_X \G\to P$ is $K^\C$-equivariant.  We call $P$ with such an action a principal $K^\C$-$\G$-bundle.  The action of $K^\C$ on $P$ descends to an action on $E$.  We summarise all the above data by referring to $E$ as a homogeneous $K^\C$-$\G$-bundle.  

Viewed as a holomorphic fibre bundle with fibre $G$, the bundle $P$ can be taken to have structure group $A\ltimes G$.  Equivariance of the action map $P\times_X \G\to P$ is equivalent to $K^\C$ acting on $P$ by $A\ltimes G$-bundle maps, meaning that $K^\C$ acts on the fibres of $P$ by elements of $A\ltimes G$.  If $P'$ is another holomorphic principal $K^\C$-$\G$-bundle, then the holomorphic group bundle $\Aut P$ with fibre $G$ and the holomorphic principal bundle $\Iso(P',P)$ with fibre $G$ and structure group bundle $\Aut P$ have induced structure groups that are complex Lie groups and they have induced $K^\C$-actions by elements of the respective structure group that make the action map $\Iso(P',P)\times_X \Aut P\to\Iso(P',P)$ equivariant.

The following equivariant parametric Oka principle is our main result.

\begin{theorem}   \label{thm:main.result}
Let $E$ be a homogeneous holomorphic $K^\C$-$\G$-bundle on a reduced Stein space $X$, where $K$ is a compact real Lie group and $\G$ is a holomorphic group $K^\C$-bundle on $X$.  Then the inclusion of the space of $K^\C$-equivariant holomorphic sections of $E$ over $X$ into the space of $K$-equivariant continuous sections is a weak homotopy equivalence.
\end{theorem}

The spaces of sections are endowed with the compact-open topology.

Let us mention three special cases of the theorem.  First, the theorem holds when $E$ is a holomorphic principal $K^\C$-$\G$-bundle.  This is in fact rather easy to prove from known results (see Remark \ref{rem:simplification}); the general case of a homogeneous bundle requires much more work.  

Next, we state the special case when the source and the target are \lq\lq uncoupled\rq\rq.  This is a parametric Oka principle for equivariant maps from a Stein $K^\C$-space to a complex homogeneous space $G/H$ with a holomorphic $K^\C$-action of a fairly general kind.  Again, let the complexification $K^\C$ of a compact real Lie group $K$ act holomorphically on a reduced Stein space $X$.  Let $G$ be a complex Lie group and $H$ be a closed complex subgroup of $G$.  Let $K^\C$ act holomorphically on $G$ in two ways: by Lie group automorphisms that preserve $H$ (call $G$ with this action $G_1$), and by biholomorphisms (call $G$ with this action $G_2$), such that the multiplication map $G_2\times G_1\to G_2$ is equivariant.  Then the second action descends to $G/H$ (left cosets).  Equivariance of the multiplication map implies that each element of $K^\C$ acts on $G_2$ by a Lie group automorphism of $G$ that preserves $H$, followed by left multiplication by an element of~$G$.

\begin{corollary}   \label{cor:uncoupled}
The inclusion of the space of $K^\C$-equivariant holomorphic maps $X\to G/H$ into the space of $K$-equivariant continuous maps is a weak homotopy equivalence.
\end{corollary}

There are many interesting special cases of this corollary.  Let us mention a few.
\begin{itemize}
\item  $G$ is reductive and acts on itself by left multiplication (with $H$ trivial).
\item  $H$ is reductive and acts on $G/H$ by left multiplication.  The geometry of such actions can be very complicated.  This is an active area of study in its own right.  As an example, take $H=\textrm{SO}(n,\C)$ to be the subgroup of $G=\textrm{SL}(n,\C)$ fixed by the holomorphic involution $A\mapsto (A^{-1})^t$.  Then $G/H$ is the space of symmetric bilinear forms on $\C^n$.
\item  The first and second actions are the same.  For example, $G$ is reductive and acts on itself by conjugation (with $H$ trivial).
\item  The first action is conjugation by elements of the normaliser $N_G(H)$, and the second action is of the form $k\cdot gH=gk^{-1}H=gHk^{-1}$.
\end{itemize}

Finally, we state the special case of no action.

\begin{corollary}   \label{cor:no.action}
Let $\G$ be a holomorphic group bundle on a reduced Stein space $X$, and let $E$ be a homogeneous holomorphic $\G$-bundle on $X$.  Then the inclusion of the space of holomorphic sections of $E$ over $X$ into the space of continuous sections is a weak homotopy equivalence.
\end{corollary}

This corollary generalises a theorem of Ramspott \cite[Satz, p.~236]{Ramspott}, which states that the inclusion induces a bijection of path components (for a bundle $E$ with a trivial structure group bundle).  When $E$ is a principal bundle (again, with a trivial structure group bundle), the corollary is implicit in work of Forster and Ramspott \cite{Forster-Ramspott}, and may perhaps be said to be implicit also in the earlier work of Grauert \cite{GrauertLiesche} and Cartan \cite{Cartan58}.  Today the corollary follows from more general results of modern Gromov-style Oka theory, for example \cite[Corollary 5.4.8]{Forstneric-book}.

Arguing as in \cite[Section 16]{Larusson}, we can prove the following purely homotopy-theoretic consequence of Theorem \ref{thm:main.result} (conversely, the corollary obviously implies the theorem).  We denote the spaces of $K$-equivariant holomorphic and continuous sections of $E$ over $X$ by $\Gamma_\O(E)^K$ and $\Gamma_\mathscr C (E)^K$, respectively.  Note that for holomorphic sections, $K$-equivariance and $K^\C$-equivariance are equivalent.

\begin{corollary}   \label{cor:homotopy-consequence}
With the same assumptions as in Theorem \ref{thm:main.result}, if $A$ is a subcomplex of a CW complex $B$, and $f:A\to\Gamma_\O(E)^K$ is continuous, then the inclusion
\[ \{\textrm{extensions }B\to\Gamma_\O(E)^K\textrm{of } f\} \hookrightarrow \{\textrm{extensions }B\to\Gamma_\mathscr C (E)^K\textrm{of } f\} \]
is a weak homotopy equivalence.

\end{corollary}

Equivariant Oka theory started with the 1995 paper of Heinzner and Kutzschebauch \cite{Heinzner-Kutzschebauch}.  The present paper and our previous paper \cite{KLS} rely heavily on \cite{Heinzner-Kutzschebauch}.  As a corollary of Theorem \ref{thm:main.result}, we obtain the following strengthening of Heinzner and Kutzschebauch's main result on the classification of principal bundles with a group action (Theorem \ref{thm:HK-classification} below).  We describe the relationship between Theorems \ref{thm:main.result} and \ref{thm:HK-strengthened} at the end of Section \ref{sec:preliminaries}.

\begin{theorem}  \label{thm:HK-strengthened}
Let the complexification $K^\C$ of a compact real Lie group $K$ act holomorphically on a reduced Stein space $X$ and on a holomorphic group bundle $\G$ on $X$.  Let $P_1$ and $P_2$ be holomorphic principal $K^\C$-$\G$-bundles on $X$.  

Every continuous $K$-isomorphism $P_1\to P_2$ can be deformed through such isomorphisms to a holomorphic $K$-isomorphism.  In fact, the inclusion of the space of holomorphic $K$-isomorphisms $P_1\to P_2$ into the space of continuous $K$-isomorphisms is a weak homotopy equivalence.
\end{theorem}

In the next section, we present the basic results that we need on generalised principal bundles: a topological equivariant local triviality theorem (Theorem \ref{thm:local.triviality}) and a homotopy invariance theorem (Theorem \ref{thm:homotopy.theorem}).  We also review the results that we need from \cite{Heinzner-Kutzschebauch}.  Section \ref{sec:proofs} consists of proofs.  In Section \ref{sec:equivar.isos}, we use Theorem \ref{thm:main.result} to strengthen an Oka principle for equivariant isomorphisms from \cite{KLS}.

\smallskip\noindent
\textit{Acknowledgement.}  We thank Michael Murray for help with the theory of generalised principal bundles.

\section{Preliminaries}  \label{sec:preliminaries}

\noindent
As in the introduction, let $X$ be a reduced Stein space and $\G$ be a holomorphic group bundle on $X$, and let $K$ be a compact Lie group with complexification $K^\C$ that acts holomorphically on $X$ and $\G$.  Let $\pi:X\to X\git K^\C$ denote the categorical quotient map.  (Here, the complexification $X^\C$ in \cite{Heinzner-Kutzschebauch} is the same as $X$.)  

To every real-analytic $K$-invariant strictly plurisubharmonic exhaustion function on $X$ (such functions exist) is associated a real-analytic subvariety $R$ of $X$ called a Kempf-Ness set.  It consists of precisely one $K$-orbit in every closed $K^\C$-orbit in $X$.  Indeed, the inclusion $R\hookrightarrow X$ induces a homeomorphism $R/K\to X\git K^\C$, where the orbit space $R/K$ carries the quotient topology.  The following result, in its original form, is due to Neeman \cite{Neeman1985}; see also \cite{Schwarz1989} and \cite{Heinzner-Huckleberry}.

\begin{theorem}  \cite[p.~341]{Heinzner-Kutzschebauch}   \label{thm:HK-retraction}
There is a real-analytic $K$-invariant strictly plurisubharmonic exhaustion function on $X$, whose Kempf-Ness set $R$ is a $K$-equivariant continuous strong deformation retract of $X$, such that the deformation preserves the closure of each $K^\C$-orbit.
\end{theorem}

In the following, we take $R$ to be a Kempf-Ness set as in this theorem.

Let $C$ be a compact Hausdorff space and $N\subset H$ be closed subsets of $C$, such that $N$ is a strong deformation retract of $C$.  We define a sheaf $\Q(R)$ of topological groups on $X\git K^\C$ as follows.  For each open subset $V$ of $X\git K^\C$, the group $\Q(R)(V)$ consists of all $K$-equivariant $NHC$-sections of $\G$ over $W=(\pi^{-1}(V)\times H)\cup((\pi^{-1}(V)\cap R)\times C)$.  By an $NHC$-section of $\G$ over $W$, we mean a continuous map $s:W\to\G$ such that:
\begin{itemize}
\item  for every $c\in C$, the map $s(\cdot,c)$ is a continuous section of $\G$ over $\pi^{-1}(V)\cap R$,
\item  for every $c\in H$, $s(\cdot,c)$ is a holomorphic section of $\G$ over $\pi^{-1}(V)$,
\item  for every $c\in N$, $s(\cdot,c)$ is the identity section of $\G$ over $\pi^{-1}(V)$.
\end{itemize}
The topology on $\Q(R)(V)$ is the compact-open topology.

The main technical result of \cite{Heinzner-Kutzschebauch} is the following.

\begin{theorem}  \cite[p.~324]{Heinzner-Kutzschebauch}   \label{thm:HK-NHC}
\begin{enumerate}
\item[(a)] The topological group $\Q(R)(X\git K^\C)$ is path connected.
\item[(b)] If $U$ is Runge in $X\git K^\C$, then the image of $\Q(R)(X\git K^\C)$ in $\Q(R)(U)$ is dense.
\item[(c)] $H^1(X\git K^\C,\Q(R))=0$.
\end{enumerate}
\end{theorem}

Next we state Heinzner and Kutzschebauch's main result on the classification of principal $K$-$\G$-bundles (called $\G$-principal $K$-bundles in \cite{Heinzner-Kutzschebauch}).

\begin{theorem}  \cite[p.~341, 345]{Heinzner-Kutzschebauch}   \label{thm:HK-classification}
{\rm (a)}  Every topological principal $K$-$\G$-bundle on $X$ is topologically $K$-isomorphic to a holomorphic principal $K^\C$-$\G$-bundle on $X$.

{\rm (b)}  Let $P_1$ and $P_2$ be holomorphic principal $K^\C$-$\G$-bundles on $X$. Let $c$ be a continuous $K$-equivariant section of $\Iso(P_1,P_2)$ over $R$. Then there exists a homotopy of continuous $K$-equivariant sections $\gamma(t)$, $t\in[0,1]$, of $\Iso(P_1,P_2)$ over $R$ such that $\gamma(0)=c$ and $\gamma(1)$ extends to a holomorphic $K$-equivariant isomorphism from $P_1$ to $P_2$.
\end{theorem}

In \cite{Heinzner-Kutzschebauch}, part (a) is only proved in the case of the structure group bundle $\G$ being a product bundle $X\times G\to X$ with a diagonal action of $K^\C$.  This is precisely the special case of \textit{ordinary} principal bundles.  On the other hand, both parts (a) and (b) are proved under the weaker assumption that only the compact group $K$ acts on the Stein space $X$, not the complexification $K^\C$.  In this sense, part (b) is more general in \cite{Heinzner-Kutzschebauch}.  Concerning part (a), the reason for the the restriction to a product group bundle with a diagonal $K$-action was that it was not known whether a general group $K$-bundle on a Stein $K$-space must extend to a $K^\C$-bundle on $X^\C$.  If one assumes, as we do in the present paper, that the group bundle $\G$ lives on a Stein $K^\C$-space and has a $K^\C$-action, then the proof of part (a) in \cite{Heinzner-Kutzschebauch} applies verbatim and yields part (a) as stated here.

A key to the proof of Theorem \ref{thm:HK-classification}(a), as well as the proof of our Theorem \ref{thm:main.result}, is a topological equivariant local triviality theorem, claimed in \cite[Remark, p.~343]{Heinzner-Kutzschebauch}.  Because of its importance we give a detailed  proof in Section~\ref{sec:proofs}.  Before stating the theorem we briefly recall some basic notions of the theory of transformation groups.  Let a compact Lie group $K$ act continuously on a topological space $X$.  A \textit{slice} at a point $x\in X$ with stabiliser $K_x$ is a locally closed $K_x$-invariant subset $S$ of $X$ containing $x$, such that the $K$-equivariant map $K\times^{K_x} S\to X$, $[k,s]\mapsto k s$, is a homeomorphism onto a $K$-invariant neighbourhood of the orbit $Kx$ of $x$.  The map or its image, the neighbourhood $K S\simeq K\times^{K_x} S$, is then called a \textit{tube} about $Kx$.  

The classical topological \textit{slice theorem} states that if $X$ is completely regular, then there is a slice at each point of $X$.  There is also a slice theorem in the smooth category such that $S$ is $K_x$-diffeomorphic to a $K_x$-module, namely $T_x X/T_x Kx$.  And in the holomorphic category, there is a slice theorem for actions of a reductive complex Lie group on a Stein space (see \cite{Snow}, \cite[\S 5.5, \S 6.3]{Heinzner1991}, \cite[p.~331]{Heinzner-Kutzschebauch}).

The setting of the holomorphic slice theorem is as follows.  Let the complexification $K^\C$ of a compact Lie group $K$ act holomorphically on a Stein space $X$ with categorial quotient $\pi:X\to X\git K^\C$.  Let $x$ be a point in a Kempf-Ness set $R$ with stabiliser $L^\C$, where $L=K_x$.  Let $V=T_xX/T_x K^\C x$ be the normal space to the orbit $K^\C x$ at $x$.  It is an $L^\C$-module.  With respect to the identification $K^\C/L^\C\simeq K^\C x$, the normal bundle $N$ of $K^\C x$ in $X$ is isomorphic to $K^\C\times^{L^\C} V$.  

\begin{theorem}   \label{thm:holomorphic.slice}
There is a $K^\C$-invariant Stein neighbourhood $U$ of the orbit $K^\C x$ in $X$, $K^\C$-equivariantly biholomorphic to a subvariety $A$ of a neighbourhood of the zero section of $N$.  The embedding $\iota:U\to N$ maps $K^\C x$ biholomorphically onto the zero section of $N$.  Moreover, $U$ can be chosen so that the following hold.
\begin{enumerate}
\item[(i)]  $U$ is saturated with respect to $\pi$.
\item[(ii)]  There are arbitrarily small $L$-invariant neighbourhoods $D$ of the origin in $V$ with $D\cap A\subset \iota(U)$, such that if we identify $D\cap A$ with a locally closed $L$-invariant subvariety of $X$ through $x$, then $K^\C(D\cap A)$ is saturated with respect to $\pi$.
\item[(iii)]  Every $L$-equivariant holomorphic map from $D\cap A$ into a complex $K^\C$-space $Y$ extends uniquely to a $K^\C$-equivariant holomorphic map $K^\C(D\cap A)\to Y$.
\end{enumerate}
\end{theorem}

A locally closed $L$-invariant subvariety of $X$ through $x\in R$ with the above properties of $D\cap A$ is called a slice at $x$ in $X$, or a \textit{Luna slice} (to acknowledge the earlier, algebraic version of the theorem due to Luna \cite{Luna}; in the literature, this name is sometimes attached to $L^\C(D\cap A)$).

A $K$-bundle $E$ on $X$ with fibre $F$ is \textit{equivariantly locally trivial} if it has the simplest local structure that we can reasonably expect, namely, locally over $X$, $E$ is induced from a trivial bundle with a diagonal action.  More precisely, every point $x$ in $X$ has a slice $S$ such that the restriction of $E$ to the tube $K S$ is $K$-equivariantly homeomorphic to the $K$-bundle $K\times^{K_x}(S\times F)\to K\times^{K_x}S$, where $K_x$ acts diagonally on $S\times F$.  Topological equivariant local triviality does not always hold.  In our theorem, the hypotheses that make it true are complex-analytic.  Note that in the parameter space $Y$, on which $K$ acts trivially, a slice at a point is simply a neighbourhood of the point.

\begin{theorem}  \label{thm:local.triviality}
Let a compact Lie group $K$ act continuously on a Stein space $X$ by biholomorphisms.  (The action $K\times X\to X$ is then automatically real-analytic.)  Let $Y$ be a topological space with a trivial $K$-action.  Let $E\to X\times Y$ be a locally trivial bundle, whose fibre $F$ is a complex manifold and whose structure group is a Lie subgroup $B$ of the biholomorphism group of $F$.  Let $K$ act continuously on $E$ by $B$-bundle maps.

Then $E$ is $K$-equivariantly locally trivial in the following sense.  Let $x\in X$ have stabiliser $L=K_x$, let $S$ be a slice for the $K$-action at $x$, and let $y\in Y$.  There is a neighbourhood $V$ of $y$ in $Y$ such that after possibly shrinking $S$, the restriction of $E$ to the $K$-invariant neighbourhood $K S\times V$ of $(x,y)$ is $K$-equivariantly homeomorphic to the $K$-bundle 
\[K\times^L((S\times V)\times F)\to K\times^L(S\times V), \]
where $L$ acts diagonally on $(S\times V)\times F$.  If $Y$ is a cube, then we can take $V=Y$.
\end{theorem}

Theorem \ref{thm:local.triviality} allows us to prove a homotopy invariance theorem.  Such a theorem is well known and fundamental in the theory of ordinary principal bundles, but, to our knowledge, does not exist in the literature for generalised principal bundles.  We give a proof in Section~\ref{sec:proofs}.

Let $X$ be a space and $E$ be a bundle of some kind on $X\times I$, where $I=[0,1]$.  We say that $E$ is \textit{isomorphic to a constant bundle} if $E$ is isomorphic, in the relevant category, to the bundle $p^*E$, where $p:X\times I\to X\times I$, $(x,t)\mapsto(x,0)$.  Then the bundles $E\vert_{X\times\{t\}}$, $t\in I$, viewed as bundles on $X$, are mutually isomorphic.

\begin{theorem}  \label{thm:homotopy.theorem}
Let a compact Lie group $K$ act real-analytically on a Stein space $X$ by biholomorphisms, and trivially on $I$.  Let $G$ be a complex Lie group and $\G$ be a topological group bundle on $X\times I$ with fibre $G$, whose structure group $A$ is a Lie subgroup of the Lie automorphism group of $G$.  Let $K$ act continuously on $\G$ by group $A$-bundle maps.

{\rm (a)}  Then $\G$ is isomorphic to a constant bundle.

{\rm (b)}  Let $P$ be a topological principal $K$-$\G$-bundle on $X\times I$.  (It is implicit that the action map $P\times_X \G\to P$ is $K$-equivariant.)  By {\rm (a)}, we may take $\G$ to be constant.  Then $P$ is isomorphic to a constant bundle.  Hence, once we identify the bundles $\G\vert_{X\times\{t\}}$, $t\in I$, with a topological group $K$-bundle $\G_0$ on $X$, the topological principal $K$-$\G_0$-bundles $P\vert_{X\times\{t\}}$, $t\in I$, are mutually isomorphic.
\end{theorem}

Finally, we use the theorem to prove the following useful result.

\begin{proposition}  \label{prp:topological.fact}
Let a compact Lie group $K$ act real-analytically on a Stein space $X$ by biholomorphisms.  Let $G$ be a complex Lie group and $\G$ be a topological group bundle on $X$ with fibre $G$, whose structure group $A$ is a Lie subgroup of the Lie automorphism group of $G$.  Let $K$ act continuously on $\G$ by group $A$-bundle maps.

Let $E$ be a topological $K$-$\G$-bundle on $X$ (not necessarily homogeneous).  The restriction map from the space of continuous $K$-sections of $E$ over $X$ to the space of continuous $K$-sections of $E$ over $R$ is a homotopy equivalence.
\end{proposition}

\begin{proof}
Let $\rho:X\to R$ be a strong deformation retraction and $\iota:R\to X$ be the inclusion.  Let $\phi:X\times I\to X$ be a homotopy from $\id_X$ to $\iota\circ\rho$ relative to $R$.  Let $P$ be the principal $K$-$\G$-bundle associated to $E$.  Theorem \ref{thm:homotopy.theorem} applied to the principal bundle $\phi^*P$ shows that the bundles $P$ and $\phi^*\iota^*P$ are $K$-isomorphic, so $E$ and $\phi^*\iota^*E$ are $K$-isomorphic as well.   Then we replace $E$ by $\phi^*\iota^* E$, note that a section of the latter is nothing but a lifting of $\rho$ by the projection $\iota^*E\to R$, and use $\phi$ to show that the maps $\iota^*$ and $\rho^*$ between the space of liftings and the space of sections of $\iota^*E=E\vert_R$ are homotopy inverse to each other.
\end{proof}

The relationship between Theorems \ref{thm:main.result} and \ref{thm:HK-strengthened} is as follows.  First we use Proposition \ref{prp:topological.fact} with $E=\Iso(P_1,P_2)$ to obtain the first statement of Theorem \ref{thm:HK-strengthened} from Theorem \ref{thm:HK-classification}(b).  This is used in the first part of the proof of Theorem \ref{thm:main.result}.  The second part of the proof of Theorem \ref{thm:main.result} then yields the second statement of Theorem \ref{thm:HK-strengthened}.  Theorem \ref{thm:HK-NHC} is a crucial ingredient in the second part of the proof of Theorem \ref{thm:main.result}.

\section{Proofs of the main theorems}  \label{sec:proofs}

\noindent
This section contains the proofs of Theorems \ref{thm:main.result},  \ref{thm:local.triviality}, and \ref{thm:homotopy.theorem}.

\subsection{Proof of the equivariant parametric Oka principle (Theorem \ref{thm:main.result})}
First we prove that the inclusion $\Gamma_\O(E)^K\hookrightarrow\Gamma_\mathscr C(E)^K$ induces a surjection of path components.  Let $P$ be the holomorphic principal $K^\C$-$\G$-bundle associated to $E$.  Take a continuous $K$-section $s$ of $E$ over $X$.  The preimage in $P$ of its image in $E$ is a topological principal $K$-$\H$-subbundle $Q$ of $P$.  We have a topological $K$-isomorphism $\sigma:Q\times^\H\G \to P$.  By Theorem \ref{thm:HK-classification}(a), $Q$ is topologically $K$-isomorphic to a holomorphic principal $K^\C$-$\H$-bundle $Q'$.  Choose a topological $K$-isomorphism $Q'\to Q$ and let $\tau:Q'\times^\H\G \to Q\times^\H\G$ be the induced isomorphism.  

By Theorem \ref{thm:HK-classification}(b) and Proposition \ref{prp:topological.fact}, the topological $K$-isomorphism $\sigma\circ\tau:Q'\times^\H\G \to P$ can be deformed to a holomorphic $K$-isomorphism over $X$.  Applying the deformation to $Q'$, viewed as a subbundle of $Q'\times^\H\G$, gives a deformation of $Q$ through topological principal $K$-$\H$-subbundles of $P$ to a holomorphic principal $K^\C$-$\H$-subbundle.  Pushing down to $E$ yields a deformation of $s$ through continuous $K$-sections of $E$ to a holomorphic section.

Now let $B$ be the closed unit ball in $\R^k$, $k\geq 1$, and let $\alpha_0:B\to \Gamma_\mathscr C(E)^K$ be a continuous map taking the boundary sphere $\partial B$ into $\Gamma_\O(E)^K$.  Choose a base point $b_0\in \partial B$.  We shall prove that there is a deformation $\alpha:B\times I\to\Gamma_\mathscr C(E)^K$ of $\alpha_0=\alpha(\cdot,0)$ with $\alpha_t(b_0)=\alpha_0(b_0)$ and $\alpha_t(\partial B)\subset\Gamma_\O(E)^K$ for all $t\in I$, and $\alpha_1(B)\subset\Gamma_\O(E)^K$.  This implies that the inclusion $\Gamma_\O(E)^K\hookrightarrow\Gamma_\mathscr C(E)^K$ induces a $\pi_{k-1}$-monomorphism and a $\pi_k$-epimorphism.

Consider the holomorphic group $K^\C$-bundle $\Aut P$ of principal $\G$-bundle automorphisms of $P$.  We seek a global $K$-equivariant $NHC$-section $\gamma_0$ of $\Aut P$ (with $C=B$, $H=\partial B$, $N=\{b_0\}$)  such that for every $b\in B$, $\gamma_0(b)$, by its left action on $E$, maps $\alpha_0(b_0)$ to $\alpha_0(b)$, over $X$ if $b\in\partial B$ but only over $R$ if $b\in B\setminus \partial B$.

\smallskip\noindent
\textbf{Claim.}  On a sufficiently small saturated neighbourhood of each point of $X$, that is, locally over $X\git K^\C$, such an $NHC$-section exists.
\smallskip

We will accept the claim for the moment, complete the proof of the theorem, and then prove the claim.

On the intersection of two such saturated neighbourhoods, two such $NHC$-sections differ by a $K$-equivariant $NHC$-section of the holomorphic group $K^\C$-bundle $\A$ of principal $\G$-bundle automorphisms of $P$ that fix $\alpha_0(b_0)$.  Gluing these local $NHC$-sections together to produce $\gamma_0$ amounts to splitting a cocycle, and the cocycle does split by Theorem \ref{thm:HK-NHC}(c) applied to~$\A$.

By Theorem \ref{thm:HK-NHC}(a) applied to $\Aut P$, we can deform $\gamma_0$ through $K$-equivariant $NHC$-sections $\gamma_t$ of $\Aut P$, $t\in I$, to the identity section.  Let $\alpha_t(b)$ be the section of $E$ obtained by letting $\gamma_t(b)$ act on $\alpha_0(b_0)$.  For $b\in B\setminus \partial B$ and $t\in(0,1)$, $\alpha_t(b)$ is only defined over $R$.  Thus we have a deformation $\alpha:B\times I\to\Gamma_\mathscr C(E\vert_R)^K$, such that $\alpha$ factors through $\Gamma_\mathscr C(E)^K$ on $B\times\{0\}$ and through $\Gamma_\O(E)^K$ on $\partial B\times I\cup B\times\{1\}$, in such a way that $\alpha_t(b_0)$ is fixed and $\alpha_1$ takes all of $B$ to $\alpha_0(b_0)$.

We could finish the proof by showing that the following commuting square has a continuous lifting.  We have renamed the spaces for convenience.
\[ \xymatrix{
A_0=\partial B\times I\cup B\times\{0,1\} \ar^{\phantom{mmmmm}\beta}[r] \ar^j[d] & \Gamma_\mathscr C(E)^K=Y \ar^p[d]  \\ A= B\times I \ar^\alpha[r] \ar@{-->}[ur] & \Gamma_\mathscr C(E\vert_R)^K=Z
} \]
In fact, it suffices to show that $\alpha$ can be deformed, keeping $\beta$ fixed and the square commuting, until a lifting exists.

As the inclusion of a subpolyhedron in a polyhedron, $j$ is a cofibration.  The precomposition maps
\[ j_Y^*:\mathscr C(A,Y)\to\mathscr C(A_0,Y), \qquad j_Z^*:\mathscr C(A,Z)\to\mathscr C(A_0,Z) \]
are Hurewicz fibrations \cite[Theorem 10]{Strom1968}.  Since the restriction map $p$ is a homotopy equivalence by Proposition \ref{prp:topological.fact}, the postcomposition maps
\[ p_*:\mathscr C(A,Y)\to\mathscr C(A,Z),  \qquad  p_*:\mathscr C(A_0,Y)\to\mathscr C(A_0,Z)  \]
are homotopy equivalences.  Consider the fibres $F_Y=(j_Y^*)^{-1}(\beta)$ and $F_Z=(j_Z^*)^{-1}(p\circ\beta)$.  By the long exact sequence of homotopy groups for a fibration, the map $p_*:F_Y\to F_Z$ is a weak homotopy equivalence; in particular it induces a surjection of path components.  Hence, $\alpha\in F_Z$ can be deformed within $F_Z$ to a map in $p_*(F_Y)$, q.e.d.

We now turn to the \textbf{proof of the claim}.  We want a lifting in the following diagram locally over each point in the categorical quotient $Q=X\git K^\C$,
\[ \xymatrix{
 & \Aut P \ar[d] \\ B\times X \ar_{\!\!\!\!\!\alpha_0}[r] \ar@{-->}[ur] &  E=P/\H
} \]
where the vertical map $\gamma \mapsto \gamma \alpha_0 (b_0)$ uses the action of $\Aut P$ on $E$.  Since $\alpha_0 (b_0)$ is a holomorphic $K^\C$-equivariant section, the map is $K^\C$-equivariant.  In fact, $\Aut P \to E $ is a locally trivial holomorphic $K^\C$-bundle with fibre $H$. 

Let us describe our aim more explicitly.  Recall that we chose a real-analytic $K$-invariant strictly plurisubharmonic exhaustion function on $X$, whose Kempf-Ness set $R$ is a $K$-equivariant continuous strong deformation retract of $X$, such that the deformation preserves the closure of each $K^\C$-orbit, so it preserves $\pi^{-1}(q)$ for each $q\in Q$.  Here, $\pi:X \to Q$ is the categorical quotient map.  For each point in $Q$, we want a neighbourhood $\Omega$ with a $K$-equivariant continuous lifting $(\partial B \times \pi^{-1} (\Omega)) \cup (B \times (\pi^{-1} (\Omega )\cap R)) \to \Aut P $, which is holomorphic and thus $K^\C$-equivariant when restricted to sets of the form $\{b\} \times \pi^{-1} (\Omega)$, $b \in \partial B$.  On $\{b_0\} \times X$, we have a lifting given by the identity section of $\Aut P$, which we wish to keep unchanged.

By pulling back the bundle $\Aut P \to E $ by $\alpha_0$, our problem transforms to finding a section, local in the above sense, of a topological $K$-bundle $Z$ over $B\times X$, which is a holomorphic $K^\C$-bundle when restricted to sets of the form $\{b\}\times X$, $b\in\partial B$.  We proceed in three steps.
\begin{enumerate}
\item  Produce a continuous section $\sigma_0$ over $B\times(\pi^{-1}(\Omega)\cap R)$.
\item  Deform $\sigma_0\vert_{\partial B\times(\pi^{-1}(\Omega)\cap R)}$ to a section that extends continuously to a section $\sigma_1$ on $\partial B\times\pi^{-1}(\Omega)$ that is holomorphic in the second variable. 
\item  Extend $\sigma_1$ continuously across $B\times(\pi^{-1}(\Omega)\cap R)$.
\end{enumerate}

Let $q\in Q$ and $x_0\in\pi^{-1}(q) \cap R$.  Let $L=K_{x_0}$ be the stabiliser of $x_0$.  Let $S_\R$ be a topological slice for the $K$-action on $X$ at $x_0$, that is, $S_\R$ is a locally closed $L$-invariant subset of $X$ containing $x_0$, such that the $K$-invariant neighbourhood $W= K S_\R $ of $x_0$ in $X$ is $K$-equivariantly homeomorphic to $K\times^L S_\R$.  By Theorem \ref{thm:local.triviality}, noting that $B$ is homeomorphic to a cube, we can shrink $S_\R$ so that  the bundle $Z$ is topologically $K$-isomorphic to 
\[ B \times (K \times^L (S_\R \times F_{(b_0, x_0) })) \to B\times (K\times^L S_\R), \] 
where the $L$-action on $S_\R \times F_{(b_0, x_0)}$ is diagonal.  Here, $ F_{(b_0, x_0)}$ is the fibre of $Z$ over $(b_0, x_0)$.  By the above, we have a $K$-section over a set containing $(b_0, x_0)$, which shows that $(b_0, x_0)$ is an $L$-fixed point.  Therefore $Z$ has an $L$-section representing the identity of $\Aut P$ in the coordinates given by the chosen trivialisation over $B \times S_\R$, which by slice theory can be extended to a continuous $K$-section $\sigma_0$ over $B\times W\simeq B \times (K\times^L S_\R)$.  Over $\{b_0\}\times W$, $\sigma_0$ represents the identity section of $\Aut P$. 

Now \cite[Corollary 1(i), (ii), p.~329]{Heinzner-Kutzschebauch} applied to the $K$-invariant subset $K x_0$ of $R$ provides a neighbourhood $\Omega$ of $q$ in $Q$ such that after replacing $W$ by $W\cap \pi^{-1}(\Omega)$, we have $\pi^{-1} (\Omega) \cap R =  W \cap R$.  Hence, $W \cap R$ is a Kempf-Ness set for $\pi^{-1} (\Omega)$.  This concludes the first step of the proof of the claim.

For the second step we use the proposition below.  Here we can replace $\partial B$ by any compact Hausdorff space $T$ with a trivial $K^\C$-action, and assume that we have a  topological $K^\C$-bundle $Z\to T \times X$ that is a holomorphic $K^\C$-bundle when restricted to sets of the form $\{t\} \times X$.  Now $X$ stands for $\pi^{-1}(\Omega)$.  The fibre $F$ of $Z$ is a complex manifold on which $K^\C$ acts by elements of a complex Lie subgroup of the biholomorphism group of $F$.  We have a $K$-invariant continuous section $\sigma_0$ over $T\times W$, where $W$ is a $K$-invariant neighbourhood of $R$.  For one point $t_0 \in T$, $\sigma_0$ extends to a  global holomorphic $K^\C$-equivariant section over $\{t_0\} \times X$.

\begin{proposition}
Every point in $Q$ has a neighbourhood $\Omega$ with a $K^\C$-equivariant continuous section $\sigma_1  : T \times \pi^{-1} (\Omega ) \to Z$, holomorphic when restricted to sets of the form $\{t\} \times \pi^{-1} (\Omega)$, along with a homotopy $\sigma_s$ of $K$-equivariant continuous sections of $Z$ over $T \times (\pi^{-1} (\Omega)\cap R)$ joining the restrictions of $\sigma_0$ and $\sigma_1$ to $T\times(\pi^{-1} (\Omega)\cap R)$.  Moreover, $\sigma_s = \sigma_0$ on $\{t_0\} \times X$ for all $s\in I$.
\end{proposition}

\begin{proof}
Let $q\in Q$ and $x_0\in\pi^{-1}(q) \cap R$.  Let $L=K_{x_0}$ be the stabiliser of $x_0$.  Let $S_\R$ be a topological slice for the $K$-action on $X$ at $x_0$.  We also have a Luna slice $S$ at $x_0$.

Note that $\sigma_0(t, x_0)$ is an $L$-fixed point for each $t\in T$.  By Proposition \ref{prp:mixed.box.coordinates} below, for each $t \in T$, there is a neighbourhood of $\sigma_0 (t, x_0)$ in $Z$, $L$-homeomorphic (over the bundle projection) to $T_0 \times V \times O$, such that the $L$-homeomorphism restricted to sets of the form $\{t\} \times V$ is holomorphic.  Here, $O$ is an $L$-neighbourhood of $0$ in the tangent space $T_{\sigma_0(t, x_0)} F_{(t, x_0)}$ to the fibre $F_{(t, x_0)}$ of the bundle (the fibre is a complex $L$-manifold), $T_0$ is a neighbourhood of $t$ in $T$, and $V$ is an $L$-invariant neighbourhood of $x_0$.  Now cover $T$ by finitely many such $T_i$ occurring in such box coordinates $T_i \times V_i \times O_i$.  Moreover, let the first $T_1$ contain our marked point $t_0$ and use box coordinates around the point $t_0$ (for later keeping the section unaltered there). 

By shrinking $V_i$, we can assume $V=V_i$ to be the same for each box, and since $T$ is compact we can moreover assume that $\sigma_0(t, x)$ is contained in the box whenever $t\in T_i$ and $x \in V$.   Also, we can arrange the finite cover $(T_i)$ of $T$ so that $T_i$ and $\bigcup_{j=1} ^{i-1}  T_j$ are separated for all $i$.  (We say that sets $A, B$ are separated if $\overline {A\setminus B}$ and $\overline{B\setminus A}$ are disjoint.) Moreover, arrange that $t_0 \in T_1 \setminus \overline T_i$ for all $i\ge 2$.

We now define a $K^\C$-equivariant section $\sigma_1$, holomorphic in the second variable, and an $L$-equivariant homotopy from $\sigma_0$ to $\sigma_1$ inductively over $\bigcup_{j=1} ^{i}  T_j \times V$ (possibly shrinking $V$ in the process).  The homotopy will have the additional property that $\sigma_s ( t, x_0)= \sigma_0(t, x_0)$.  On $T_1$, define in box coordinates $\tilde\sigma_1 (t, x) = \sigma_0(t_0, x)$, $t \in T_1$, $x \in V$, which is $L$-equivariant and holomorphic where it should be (here we use the fact that $\sigma_0(t, x)$ is contained in the box whenever $t\in T_1$). 

Since Luna slices can be shrunk, we may assume that $S \subset V\cap W$.  We can extend the restriction $\tilde\sigma_1 \vert_S$ to a continuous $K^\C$-equivariant section over $T_1 \times  K^\C S$, which is holomorphic when restricted to sets of the form $\{t\} \times \pi^{-1} (\Omega) $  and where  $K^\C S = \pi^{-1} (\Omega)$ is a saturated neighbourhood of $x_0$.  Call the extension $\sigma_1$.  In the box coordinates associated to $T_1$, we can define a homotopy $\sigma_s = s \sigma_1   + (1-s) \sigma_0$ over $T_1 \times S_\R$, which consists of $L$-equivariant sections since the $L$-action on $O_i$ is linear. This construction gives $\tilde\sigma_s (t_0, x) = \sigma_0(t_0, x)$ for $x \in S_\R$, $s\in I$, so the section over the marked point $t_0$ is not altered.

Assume that the homotopy (call it $\sigma_s^{i-1}$) exists on $\bigcup_{j=1} ^{i-1}  T_j$.  On $T_i$, define in box coordinates $\tilde\sigma_1 (t, x) = \sigma_0(t, x_0)$, $t \in T_i$, $x \in V$, which is $L$-equivariant and holomorphic where it should be.  (Here we use the fact that  $\sigma_0(t, x_0)$ is contained in the box whenever $t\in T_i$.)  Again using the Luna slice $S$, we can extend the restriction $\tilde\sigma_1 \vert_S$ to a continuous $K^\C$-equivariant section over $T_i \times  K^\C S$ , which is holomorphic when restricted to sets of the form $\{t\} \times \pi^{-1} (\Omega)$ and where  $K^\C S = \pi^{-1} (\Omega)$ is a saturated neighbourhood of $x_0$.  Call the extension $ \tilde\sigma_1 $.  By $K$-equivariance and since $x_0 \in S$, we have $\tilde\sigma_1( t, k x_0)= \sigma_0(t, k x_0)$ for $k \in K$.  Moreover, in the box coordinates associated to $T_i$, we can define a homotopy $\tilde \sigma_s = s \tilde\sigma_1 + (1-s) \sigma_0$ over $T_i \times S_\R$, which consists of $L$-equivariant sections since the $L$-action on $O_i$ is linear.

Shrinking $V$, we can furthermore assume that on $T_i \cap \bigcup_{j=1} ^{i-1} T_j$, the whole homotopy $\tilde\sigma_s$ is in the box coordinates associated to $T_i$.  By separation, there is a continuous function $\xi: T \to [0, 1]$ that equals $1$ on $\bigcup_{j=1} ^{i-1} T_j \setminus T_i$ and equals $0$ on $T_i \setminus \bigcup_{j=1} ^{i-1} T_j  $.  

The extended homotopy on $\bigcup_{j=1} ^{i} T_j$ is defined by $\xi (t) \sigma_s^{i-1} (t, x) + (1-\xi (t)) \tilde \sigma_s (t,x)$, where we use the box coordinates associated to $T_i$, that is, we look at sections as maps to $O_i$.  Since the action on $O_i$ is linear, the defined homotopy consists of $L$-equivariant sections.  The  same formula defines $\sigma_1(t,x) = \xi (t) \sigma_1 ^{i-1} (t, x) + (1-\xi (t)) \tilde \sigma_1  (t,x)$ on $\bigcup_{j=1} ^{i} T_j\times \pi^{-1} (\Omega) $.  We have $\sigma_1  ( t, k x_0)= \sigma_0(t, k x_0)$ for $k \in K$. By the additional assumption on $t_0$, we have $\xi (t_0) =1$. Therefore the homotopy is still not changed over $t_0$.  After finitely many steps the construction is complete.

The $K$-invariant neighbourhood $\pi^{-1} (\Omega) \cap W$ of $x_0$ equals $K (\pi^{-1} (\Omega) \cap S_\R)$.  By shrinking $\Omega$, we can therefore assume that the whole homotopy $\sigma_s (t,z)$ from  $\sigma_0(t, z)$ to $\sigma_1(t,z)$ is defined for all $t\in T$ and $z \in \pi^{-1} (\Omega) \cap S_\R$. 

This homotopy extends to a $K$-equivariant homotopy between the restrictions of $\sigma_0$ and $\sigma_1$ to $\pi^{-1} (\Omega) \cap W$.  Now \cite[Corollary 1(i), (ii), p.~329]{Heinzner-Kutzschebauch} applied to the $K$-invariant subset $K x_0$ of $R$ allows us to shrink $W$ and $\Omega$ further, so that $\pi^{-1} (\Omega) \cap R =  W \cap R$.  In other words, $W \cap R$ is a Kempf-Ness set for $\pi^{-1} (\Omega)$.  Thus we have constructed $\sigma_1 $ as desired, together with a homotopy of $K$-equivariant continuous sections defined over $T \times (W\cap R)$ connecting the restrictions of $\sigma_0$ and $\sigma_1$. 
\end{proof}

We now turn to the third and final step of the proof of the claim.  We have a continuous section $\sigma_1$ of the bundle $Z$ over $\partial B\times\pi^{-1}(\Omega)$.  We wish to extend it continuously across $B\times(\pi^{-1}(\Omega)\cap R)$.  We know that $\sigma_1\vert_{\partial B\times(\pi^{-1}(\Omega)\cap R)}$ is homotopic to a section that extends continuously to $B\times(\pi^{-1}(\Omega)\cap R)$.  We need a $K$-equivariant version of the deformation-invariance of extendability.  For convenience, set $M=B\times(\pi^{-1}(\Omega)\cap R)$ and $N=\partial B\times(\pi^{-1}(\Omega)\cap R)$.  We need a $K$-equivariant continuous lifting in the following square,
\[ \xymatrix{
(N\times I)\cup (M\times\{0\}) \ar[r] \ar^j[d] & Z\vert_M \ar^p[d] \\ M\times I \ar@{-->}[ur] \ar[r] & M
} \]
so we need to show that the inclusion $j$ is an acyclic cofibration and the bundle projection $p$ is a fibration in the appropriate model structure.  By basic equivariant homotopy theory \cite[Section III.1]{Mandell-May}, it suffices to observe the following.
\begin{itemize}
\item  For every closed subgroup $L$ of $K$,
\[\begin{CD}
\big((N\times I)\cup (M\times\{0\})\big)^L = \big((\partial B\times I)\cup(B\times\{0\})\big)\times (\pi^{-1}(\Omega)\cap R)^L  \\ @VVjV \\ (M\times I)^L=(B\times I)\times (\pi^{-1}(\Omega)\cap R)^L
\end{CD} \]
is a homotopy equivalence.  This is obvious.
\item  $j$ is a relative $K$-cell complex.  Indeed, as a $K$-invariant real-analytic subvariety of the complex $K$-space $\pi^{-1}(\Omega)$, $\pi^{-1}(\Omega)\cap R$ has a $K$-equivariant triangulation \cite[Theorem~B]{Illman}.
\item  For every closed subgroup $L$ of $K$, 
\[ (Z\vert_M)^L\overset{p}\longrightarrow M^L=B\times (\pi^{-1}(\Omega)\cap R)^L\] 
is locally trivial.  Namely, $K$ acts on the fibre of the bundle $\Aut P \to E$ through a structure group of biholomorphisms.  Applying Proposition \ref{prp:locally.trivial} below with $T$ a singleton and $L$ in place of $K$, we see that near each point in $E^L$, $\Aut P \to E$ is $L$-homeomorphic to a trivial bundle with a diagonal $L$-action.  Hence, $(\Aut P)^L\to E^L$ is locally trivial.
\end{itemize}

The proof of Theorem \ref{thm:main.result} is now complete.

\begin{remark}  \label{rem:simplification}
The proof of Theorem \ref{thm:main.result} simplifies drastically when $E=P$ is a principal $K^\C$-$\G$-bundle.  The first part of the proof of Theorem \ref{thm:main.result} is the same (with $\H$ trivial) and shows that the inclusion $\iota:\Gamma_\O(P)^K \hookrightarrow \Gamma_\mathscr C(P)^K$ induces a surjection of path components.  Let $r:\Gamma_\mathscr C(P)^K \to \Gamma_\mathscr C(P\vert_R)^K$ be the restriction map.  By Proposition \ref{prp:topological.fact}, $r$ is a homotopy equivalence.  Thus it suffices to show that $r\circ\iota$ induces a $\pi_{k-1}$-monomorphism and a $\pi_k$-epimorphism for every $k\geq 1$, for every base point in $\Gamma_\O(P)^K$.

Let $B$ be the closed unit ball in $\R^k$, $k\geq 1$, and let $\alpha_0:B\to \Gamma_\mathscr C(P\vert_R)^K$ be a continuous map taking the boundary sphere $\partial B$ into $\Gamma_\O(P)^K$ (more precisely, $\alpha_0\vert_{\partial B}$ factors through $r\circ\iota$).  Choose a base point $b_0\in \partial B$.  It suffices to prove that there is a deformation $\alpha:B\times I\to\Gamma_\mathscr C(P\vert_R)^K$ of $\alpha_0=\alpha(\cdot,0)$, keeping $\alpha_t(b_0)$ fixed and $\alpha_t(\partial B)$ in $\Gamma_\O(P)^K$, such that $\alpha_1$ takes all of $B$ into $\Gamma_\O(P)^K$.

For every $b\in B$, let $\gamma_0(b)$ be \textit{the} section of $\G$ with
\[ \alpha_0(b)(x)=\alpha_0(b_0)(x)\cdot \gamma_0(b)(x) \]
for all $x\in X$ if $b\in\partial B$, but only for $x\in R$ if $c\in B\setminus \partial B$.  Then $\gamma_0$ is a global $K$-equivariant $NHC$-section of $\G$ (with $C=B$, $H=\partial B$, $N=\{b_0\}$).  By Theorem \ref{thm:HK-NHC}(a), we can deform $\gamma_0$ through $K$-equivariant $NHC$-sections $\gamma_t$ of $\G$, $t\in I$, to the identity section.  Now let $\alpha_t(b)=\alpha_0(b_0)\cdot\gamma_t(b)$.  Then $\alpha$ is as desired; $\alpha_1$ even takes all of $B$ to $\alpha_0(b_0)$.
\end{remark}

\subsection{Proof of the equivariant local triviality theorem (Theorem \ref{thm:local.triviality})}
Let $K$ be a compact Lie group.  Let $Y$ be a completely regular space with a continuous $K$-action, so the topological slice theorem applies.  Let $p: E\to Y$ be a topological $K$-bundle with a complex manifold $F$ as its fibre.  Assume that the structure group $A$ of $E$ is a complex Lie subgroup of the biholomorphism group of $F$, and that $K$ acts through $A$.  Let $v_0 \in E$ be a $K$-fixed point. Then $y_0=p (v_0) \in Y$ is a $K$-fixed point and $F_{y_0}=p^{-1}(y_0)$ is a complex $K$-manifold.

\begin{proposition} \label{prp:box.coordinates} 
There exists a $K$-neighbourhood $U$ of $v_0$ in $E$ and a $K$-neighbourhood $W$ of $v_0$ in $F_{y_0}$ such that the $K$-map $p\vert_U : U \to p(U) $ is topologically $K$-isomorphic to the $K$-product bundle $p(U) \times W \to p(U)$. 
\end{proposition}  

\begin{proof}
Since the statement is local around $v_0$, we may assume that the bundle $E$ is topologically trivial, that is, a product bundle $U_0  \times F$ (but not necessarily with a diagonal action).  Now $K$ acts by a representation $\rho : K \to \GL(V)$ on $V= T_{v_0} F_{y_0}$, and there is a $K$-neighbourhood $W$ of $v_0$ in $F_{y_0}$ that is $K$-equivariantly biholomorphic to a $K$-neighbourhood $D$ of the origin $0$ in $V$ by a biholomorphism $\alpha : W \to D$.  Thus in a $K$-neighbourhood $U$ of $v_0$ in $E$, we can consider the map $\phi : U \to V$ with $\phi (v_0) = 0$ given by
\[ U \hookrightarrow U_0 \times W \overset{\pi_2}\to W \overset\alpha\to D \hookrightarrow V.\]
Since $K$ acts linearly on $V$, we can average $\phi$ over the compact group $K$: 
\[ \tilde \phi (y, w) = \int_K \rho (k) \phi (k^{-1} (y, w))\, dk. \]
This map is continuous, and when restricted to a fibre $\{y\}\times W$, it is a holomorphic map to $V$.  Indeed, by assumption, each $k \in K$ maps the fibre over $y$ biholomorphically to the fibre over $k y$, so the restriction of the integrand $\rho (k) \phi (k^{-1} (y, w))$ is a holomorphic map from a fibre $\{y\}\times W$ to $V$. Since holomorphic maps to a vector space form a Fr\'echet space, the average $\tilde \phi$ is a  holomorphic map from the fibre $\{y\}\times W$ to $V$.  Moreover the restriction of $\tilde \varphi$ to $\{y_0\}\times W$ equals the biholomorphism $\alpha$. Since the Taylor coefficients of $\tilde\varphi\vert_{\{y\} \times W}$ at $(y, v_0)$ depend continuously on $y$,  the restriction of $\tilde\varphi$ to $\{y\}\times W$ is a local biholomorphism when restricted to fibres in a $K$-neighbourhood  of $v_0$.  After shrinking if necessary, the map $(p,\tilde\phi) : U \to U_0  \times W$ is a continuous, open, and  injective $K$-equivariant map, that is, a $K$-equivariant homeomorphism onto its image $\tilde \phi (U) \subset U_0 \times D$.  Going back with $\alpha^{-1}$ from $D$ to $W$ (and shrinking if necessary) completes the proof. 
\end{proof}

The holomorphic version of Proposition \ref{prp:box.coordinates} is \cite[Proposition 2, p.~332]{Heinzner-Kutzschebauch}.  What we need here is a holomorphic version with continuous dependence on a parameter.  Let $X$ be a Stein space with an action of $K$ by biholomorphisms.  Let $T$ be a topological space with a trivial $K$-action.  Let $E \overset{p}\to Y= T\times X $ be a topological $K$-bundle with a complex manifold $F$ as its fibre, such that $p$ is a holomorphic $K$-bundle when restricted to sets of the form $\{t\} \times X$.  In other words, we have a continuous family of holomorphic $K$-bundles over the Stein space $X$.  Here $K$ acts through a complex Lie subgroup $A$ of the biholomorphism group of $F$, and $E$ is defined by a cocycle of continuous maps, each defined on a set of the form $T_0\times V \to A$, where $T_0\subset T$ and $V\subset X$ are open, and holomorphic when restricted to sets of the form $\{t\} \times V$.

Let $v_0 \in E$ be a $K$-fixed point.  Then $(t_0, x_0) = p (v_0) \in Y$ is a $K$-fixed point and $F_{(t_0,x_0)}$ is a complex $K$-manifold.

\begin{proposition} \label{prp:mixed.box.coordinates} 
There is a $K$-neighbourhood $U$ of $v_0$ in $E$ and a $K$-neighbourhood $W$ of $v_0$ in $F_{(t_0,x_0)}$ such that the $K$-map $p\vert_U: U \to p(U)=T_0\times V   $ is topologically $K$-isomorphic to the $K$-product bundle $p(U) \times W \to p(U)$, in such a way that the $K$-isomorphism restricted to sets of form $\{t\} \times V$ is holomorphic. Here, $W$ is a $K$-neighbourhood of $0$ in the tangent space $T_{\sigma (t_0, x_0)} F_{(t_0, x_0)}$, $T_0$ is a neighbourhood of $t_0$ in $T$, and $V$ is a $K$-invariant neighbourhood of $x_0$.
\end{proposition}  

\begin{proof}  The proof is similar to the proof of Proposition \ref{prp:box.coordinates}.  The bundle $E$ is topologically trivial over a set of the form $U_0= T_0\times V$, that is, it is topologically isomorphic to a product bundle $U_0 \times F$, such that the isomorphism  is holomorphic over sets of the form $\{t\} \times V$.  The restriction of $\phi$ in the earlier proof, not only to a fibre $\{x\}\times W$, but also to a set of the form $\{t\} \times V \times W$ is holomorphic.  This implies that the averaged map $\tilde \phi$ on a set of the form $V \times W$ is holomorphic.
\end{proof}

Note that Proposition \ref{prp:mixed.box.coordinates} includes \cite[Proposition 2, p.~332]{Heinzner-Kutzschebauch} as the special case when $T$ is a point.

\begin{corollary}  \label{cor:section}
Let $p: E \to Y$ be a topological $K$-bundle as in Proposition \ref{prp:box.coordinates} and let  $v_0 \in E$ be a $K$-fixed point. Then there is a $K$-invariant neighbourhood  $U_0 $ of $y_0 = p (v_0)$ in $X$ and a $K$-equivariant continuous section $s: U_0 \to E$ with $s(y_0)= v_0$. 
\end{corollary}

\begin{proof}
Take $s$ to be the constant section $y \mapsto (y, v_0)$ in the local coordinates $\Omega \times W$ provided by Proposition \ref{prp:mixed.box.coordinates}.
\end{proof}

We turn again to a parametric situation.  Let $X$ be a completely regular space with a continuous $K$-action (for example a Stein space with a $K$-action by biholomorphisms) and let $T$ be a topological space with a trivial $K$-action.  Let $E \overset p \to T\times X $ be a topological $K$-bundle on $T\times X$ with a complex manifold $F$ as its fibre, such that $E$ is defined by a cocycle of continuous maps into a complex Lie subgroup $A$ of the biholomorphism group of $F$, and $K$ acts by elements of $A$.

Let $x\in X$, let $L= K_x$ be the stabiliser of $x$, and let $S$ be a local topological slice for the $K$-action on $X$ at $x$.  Thus $S$ is an $L$-invariant subset of $X$ containing $x$, together with a $K$-equivariant homeomorphism $ U \to K\times^L S$ from a $K$-neighbourhood $U= K S$ of $x$ in $X$. Then, for each $t\in T$, $T\times S$ is a topological slice for the $K$-action on $T\times X$ at $(t,x)$.  If $X$ is a Stein $K$-space, a real-analytic slice $S$ can be constructed by locally $K$-equivariantly embedding $X$ into a $K$-module and intersecting a slice for the linear $K$-action with $X$.  For each $t\in T$, the fibre $F_{(t,x)}$ is by assumption a complex $L$-manifold. 

\begin{proposition}   \label{prp:locally.trivial}
The bundle $E$ is equivariantly locally trivial, that is, each $t_0 \in T$ has a neighbourhood $T_0$ in $T$ such that after possibly shrinking $S$,  the restriction of $E$ to the $K$-neighbourhood $T_0 \times U$ of $(t_0, x)$, where $U= K S$,  is $K$-homeomorphic to the $K$-bundle 
\[T_0 \times (K\times ^L (S \times F_{(t_0, x)}))  \to T_0\times(K\times^L S), \] 
where the $L$-action on $S \times F_{(t_0, x)}$ is diagonal.
\end{proposition}

\begin{proof}
After shrinking $S$ and choosing $T_0$ small enough, we can assume that $E$ is trivial over $ T_0 \times S$.  The $L$-action on $E\vert_{T_0 \times S}$ is given by 
\[L \times (T_0 \times S)\times F \to (T_0 \times S)\times F, \quad (l, t, s, f) \mapsto (t, l s, g(l, t, s) (f)), \] 
where $g : L \times T_0 \times S \to A$ is a continuous map satisfying the functional equations for an action.  As already mentioned,  $F_{(t_0, x)}$ is an $L$-manifold with the action $ (l, f) \mapsto  g(l, t_0, x) (f) $, where $g (l , t_0, x)$ defines a group homomorphism $L \to A$.

Consider the locally trivial bundle of isomorphisms from $E\vert_{T_0 \times S}$ (with the $L$-action already described) to $T_0 \times S \times F_{(t_0, x)}$.  Both bundles are trivial (but have different $L$-actions as described above).  The isomorphism bundle is a trivial bundle of the form $(T_0 \times S) \times  A$ and carries an $L$-action by pre- and postcomposition 
\[ l (t, s, z) = (t, l s, g(l,t, s) z g(l^{-1}, t_0, x)). \]
Since $((t_0, x), e)$ is a fixed point ($e$ being the identity element of $A$), Corollary \ref{cor:section} implies that there is a continuous $L$-section $\sigma (t,s)$ of the isomorphism bundle, which, after shrinking, can be assumed to be defined over $T_0 \times S$.  This section extends to a continuous $K$-equivariant section of the $K$-bundle of isomorphisms from $E\vert_{ T_0 \times U}$ to $T_0\times (K \times^L (S\times F_{(t_0, x)}))$.  The extended section gives the desired $K$-homeomorphism from $E\vert_{T_0\times U}$ to $T_0 \times (K\times ^L (S \times F_{(t_0, x)}))$.
\end{proof}

If $T$ is a cube $[0, 1]^n$, then we can prove local triviality over all of $T$.  

\begin{proposition}   \label{prp:cube}
Let $ x\in X$ and  $t_0 \in T=[0,1]^n$.  After possibly shrinking $S$, the restriction of $E$ to the $K$-neighbourhood $T \times KS$ of $(t_0,x)$ is $K$-homeomorphic to the $K$-bundle 
\[ T \times (K\times ^L (S \times F_{(t_0, x)}))  \to T\times(K\times^L S), \] 
where the $L$-action on $S\times F_{(t_0, x)}$ is diagonal.
\end{proposition}

\begin{proof} We use induction and apply the same argument as in the proof of Lemma \ref{lem:one} below.
\end{proof}

Propositions \ref{prp:locally.trivial} and \ref{prp:cube} now imply Theorem \ref{thm:local.triviality}.

\subsection{Proof of the homotopy invariance theorem (Theorem \ref{thm:homotopy.theorem})}
Let $K$ be a compact Lie group which acts continuously on the paracompact Hausdorff space $X$.  By the topological slice theorem, there is a slice at each point of $X$.  Let $\G$ be a  group bundle over $X$ with fibre $G$ and structure group $A$ such that $K$ acts on $\G$ by group $A$-bundle maps.  Recall the definition of equivariant local triviality from Section~\ref{sec:preliminaries}.  We assume that the $K$-action on $\G$ is equivariantly locally trivial. Let $K$ act trivially on $I$ and let $p\colon P\to X\times I$ be a principal $K$-$(\G\times I)$-bundle which is equivariantly locally trivial.  Of course, the action map $P\times_{X\times I}(\G\times I)\to P$ is $K$-equivariant.  Let $P_t$ denote the restriction of $P$ to $X\times\{t\}$, $t\in I$.  Then $P_t$ is naturally a principal $K$-$\G$-bundle.

\begin{theorem}   \label{thm:constantG}
Let $K$, $\G$, and $p\colon P\to X\times I$ be as above. Then the principal $K$-$(\G\times I)$-bundles $P$ and $P_1\times I$ are isomorphic. In particular, the principal $K$-$\G$-bundles $P_t$, $t\in I$, are mutually isomorphic.
\end{theorem}

We say that $P$ is isomorphic to a product on $X\times [a,b]$ if $P$ is isomorphic to $Q\times[a,b]$, where $Q$ is a principal $K$-$\G$-bundle on $X$.

\begin{lemma}   \label{lem:one}
Suppose that $P$ is isomorphic to a product on $X\times[0,\tfrac 1 2]$ and $X\times [\tfrac 1 2,1]$. Then $P$ is isomorphic to a product on $X\times I$.
\end{lemma}

\begin{proof}
Let $Q$ and $Q'$ be the bundles on $X$ corresponding to the two product structures.  Then $Q$ and $Q'$ are isomorphic via a $K$-equivariant isomorphism $\phi$ of principal $K$-$\G$-bundles.  Changing the trivialisation of $P$ on $X\times[\tfrac 1 2,1]$ by $\phi\inv$ we can glue the two product structures to give a product structure for $P$ on $X\times [0,1]$.
\end{proof}

\begin{proof}[Proof of Theorem \ref{thm:constantG}]
Let $(x,t)\in X\times I$.  Then a slice at $(x,t)$ has the form $S_t\times I_t$ where $S_t\subset X$ is an $H=K_x$-stable neighbourhood of $x$ and $I_t$ is a neighbourhood of $t$ in $I$.  Since $P$ is equivariantly locally trivial, we may assume  that $P\vert _{S_t\times I_t}\simeq P\vert _{S_t\times\{t\}}\times I_t$.  It follows that $P\vert _{U_t\times I_t}\simeq P\vert _{U_t\times\{t\}}\times I_t$, where $U_t=K\times^HS_t$.  Now we can cover $I$ by finitely many $I_{t_j}$ such that $P$ is isomorphic to a product on $S\times I_{t_j}$ for all $j$, where $S=\bigcap S_{t_j}$.  Using Lemma \ref{lem:one}, we see that $P$ is isomorphic to a product on $(K\times^HS)\times I$.  We may also assume that $P\vert _{S\times I}\simeq \G\vert _S\times I$.

We can find a locally finite cover $(U_j)_{j\in J}$ of $X$ by $K$-invariant open sets of the form $K\times^{H_j}S_j$ such that we have $H_j$-isomorphisms $P\vert _{S_j\times I}\simeq\G\vert _{S_j}\times I$.  Let $\rho_j: X\to I$ be continuous and $K$-invariant with support in $U_j=KS_j$ and $\sup_j \rho_j(x)=1$ for all $x\in X$.  For each $j\in J$, define $f_j:\G(S_j)\times I\to \G(S_j)\times I$ by $f_j(g(x),t)=(g(x),\max\{\rho_j(x),t\})$.  Here $g(x)$ is in the fibre of $\G$ at $x$.  Then each $f_j$ induces an automorphism of $\G(S_j)\times I$, hence an automorphism of $P\vert _{S_j}\times I$ and an automorphism $h_j$ of $P\vert _{U_j\times I}$.
  
Pick a total ordering on $J$ and for each $y\in P$ define $h(y)\in P$ by composing, in order, the maps $h_j$ for the finitely many $j$ such that the first component of $p(y)$ lies in $U_j$.  Then $h:P\to P$ is continuous since the open cover $(U_j)$ is locally finite.  By construction, the image of $h$ lies in $p\inv(X\times\{1\})$ and $h$ gives the isomorphism of $P$ with $P_1\times I$.
\end{proof}

Now let $\G$ be a group $K$-bundle over $X\times I$.  Then $\G_t$ depends upon $t\in I$.  Using the same techniques as above one proves the following result.

\begin{theorem}   \label{thm:nonconstantG}
Let $\G$ be a group $K$-bundle over $X\times I$. Then $\G\simeq \G_1\times I$.
\end{theorem}

Theorem \ref{thm:homotopy.theorem} is now a consequence of Theorems \ref{thm:local.triviality}, \ref{thm:constantG}, and \ref{thm:nonconstantG}.

\section{Equivariant isomorphisms}   \label{sec:equivar.isos}

\noindent
Theorem \ref{thm:main.result} can be used to strengthen one of the main results of our previous paper \cite{KLS}.  There, we considered generic Stein $G$-manifolds $X$ and $Y$ that are locally $G$-biholomorphic over a common categorical quotient $Q$, where $G$ is a reductive complex Lie group.  Genericity is defined in \cite[p.~194]{KLS}; as explained in \cite[Remark 5]{KLS}, generic actions really are generic in a reasonable sense.  We showed that the obstruction to $X$ and $Y$ being globally $G$-biholomorphic over $Q$ is topological and established several sufficient conditions for it to vanish.

We defined a $G$-homeomorphism $X\to X$ to be \textit{special} if it is of the form $x\mapsto \gamma(x)\cdot x$ for some continuous $G$-map $\gamma:X\to G$, where $G$ acts on the target $G$ by conjugation; $\gamma$ is then unique.  We called a $G$-homeomorphism $\psi:X\to Y$ \textit{special} if for some open cover $(U_i)$ of $Q$ and $G$-biholomorphisms $\phi_i:\pi_X^{-1}(U_i)\to\pi_Y^{-1}(U_i)$ over $U_i$ (meaning that they descend to the identity map of $U_i$), the $G$-homeomorphisms $\phi_i^{-1}\circ \psi$ of $\pi_X^{-1}(U_i)$ are special.  Here, $\pi_X:X\to Q$ and $\pi_Y:Y\to Q$ are the quotient maps.  Genericity implies that every $G$-biholomorphism $X\to Y$ over $Q$ is special.  In fact, for each open subset $U$ of $Q$, there is a bijective correspondence between $G$-biholomorphisms $\psi$ of $\pi_X^{-1}(U)$ over $U$ and holomorphic $G$-maps $\gamma:\pi_X^{-1}(U)\to G$, where $G$ acts on the target $G$ by conjugation, given by $\psi(x)=\gamma(x)\cdot x$ \cite[Lemma 6]{KLS}.  The definition of a special $K$-homeomorphism $X\to Y$, where $K$ is a maximal compact subgroup of $G$, is evident.

We defined \textit{strong} $G$-homeomorphisms in \cite[p.~210]{KLS}.  They are holomorphic and hence algebraic on each fibre.  The full definition is somewhat involved, so we shall not recall it, but it is natural, whereas special $G$-homeo\-morphisms only play an auxiliary role.  One of our Oka principles for equivariant isomorphisms \cite[Theorem 22]{KLS} states that every strong $G$-homeomorphism $X\to Y$ can be deformed, through strong $G$-homeomorphisms, to a special strong $G$-homeomorphism.  We then argued that the existence of a special $G$-homeomorphism, or merely a special $K$-homeomorphism, implies the existence of a $G$-biholomorphism, without establishing that the former can be deformed to the latter.

Let $\G$ be the holomorphic group $G$-bundle $X\times G$ over $X$, with $G$ acting on the fibre $G$ by conjugation.  Continuous $K$-sections and holomorphic $G$-sections of $\G$ correspond to special $K$-homeomorphisms and $G$-biholomorphisms $X\to X$ over $Q$, respectively.  Special $K$-homeomorphisms and $G$-biholomorphisms $X\to Y$ over $Q$ correspond to continuous $K$-sections and holomorphic $G$-sections, respectively, of a certain principal $G$-$\G$-bundle that is locally isomorphic to $\G$ over $Q$.  Theorem \ref{thm:main.result} now yields the following result.

\begin{theorem}   \label{thm:old.work.strengthened}
Let $G$ be a reductive complex Lie group.  Let $K$ be a maximal compact subgroup of $G$.  Let $X$ and $Y$ be generic Stein $G$-manifolds, locally $G$-biholomorphic over a common quotient $Q$.  Every strong $G$-homeomorphism $X\to Y$ over $Q$ can be deformed, through $K$-homeomorphisms over $Q$, to a $G$-biholomorphism.  
\end{theorem}

The deformation starts through strong $G$-homeomorphisms and continues through special $K$-homeomorphisms.  It is reasonable to conjecture that the inclusion of the space of $G$-biholomorphisms into the space of strong $G$-homeomorphisms is a weak homotopy equivalence.  We leave the consideration of this question for another day.

\bibliographystyle{amsalpha}
\bibliography{Oka.paperbib}

\end{document}